\documentclass[11pt]{amsart}
\usepackage{amsmath,amsthm,amsfonts,amssymb,amscd,flafter,pinlabel}
\usepackage[mathscr]{eucal}
\usepackage{graphics}
\usepackage[all]{xy}

\usepackage[usenames,dvipsnames]{color}
\usepackage{subfigure}
\usepackage{graphicx}

\usepackage{pgf}
\usepackage{tikz}
\usetikzlibrary{arrows,automata}
\usepackage[latin1]{inputenc}

\usepackage[colorlinks=true, urlcolor=NavyBlue, linkcolor=NavyBlue, citecolor=NavyBlue, pdftitle={Contact Invariants in Heegaard Floer Theory}, pdfauthor={John A. Baldwin, David Shea Vela-Vick}, pdfsubject={Contact Invariants}, pdfkeywords={Open book, transverse braid, Heegaard Floer homology, 57M27, 57R58}]{hyperref}

\newtheorem{theorem}{Theorem}[section]
\newtheorem{corollary}[theorem]{Corollary}

\newtheorem{question}[theorem]{Question}

\newtheorem{lemma}[theorem]{Lemma}

\theoremstyle{definition}
\newtheorem{definition}{Definition}[section]

\theoremstyle{definition}
\newtheorem{remark}[theorem]{Remark}

\newcommand\ssm{\smallsetminus}

\newcommand\bs{\boldsymbol}

\newcommand{\Z}{\ensuremath{\mathbb{Z}}}
\newcommand{\T}{\ensuremath{\mathbb{T}}}
\newcommand\HFh{\mathit{\widehat{HF}}}

\newcommand\CFh{\mathit{\widehat{CF}}}

\newcommand\HFKh{\mathit{\widehat{HFK}}}

\newcommand\CFKh{\mathit{\widehat{CFK}}}

\newcommand\Sym{{\rm {Sym}}}

\newcommand\ob{\mathfrak{ob}}

\newcommand\rk{\textrm{rk}\,}

\newcommand{\N}{\mathbb{N}}

\newcommand{\F}{\mathbb{F}}

\newcommand{\id}{\operatorname{id}}

\begin{document}

\title[{A note on the knot Floer homology of fibered knots}]{A note on the knot Floer homology of fibered knots}

\author[John A. Baldwin]{John A. Baldwin}
\address{Department of Mathematics \\ Boston College}
\email{john.baldwin@bc.edu}
\urladdr{\href{https://www2.bc.edu/john-baldwin/}{https://www2.bc.edu/john-baldwin/}}

\author[David Shea Vela-Vick]{David Shea Vela--Vick}
\address{Department of Mathematics \\ Louisiana State University}
\email{shea@math.lsu.edu}
\urladdr{\href{http://www.math.lsu.edu/~shea}{http://www.math.lsu.edu/\~{}shea}}

\thanks{JAB was partially supported by NSF Grant DMS-1406383 and NSF CAREER Grant DMS-1454865.\\
\indent DSV was partially supported by NSF Grant DMS-1249708 and Simons Foundation Grant 524876.}

\maketitle


\begin{abstract}
We prove that the knot Floer homology of a fibered knot is nontrivial in its next-to-top Alexander grading. Immediate applications include new proofs of Krcatovich's  result that knots with $L$-space surgeries are prime and  Hedden and Watson's   result that the rank of knot Floer homology detects the trefoil among knots in the 3--sphere. We also generalize the latter result, proving a similar theorem for nullhomologous knots in any 3--manifold. We note that our    method of  proof  inspired Baldwin and Sivek's recent proof that Khovanov homology detects the trefoils. 
As part of this work, we also introduce a numerical refinement of the Ozsv{\'a}th-Szab{\'o} contact invariant. This refinement was the inspiration for  Hubbard and Saltz's  annular refinement of Plamenevskaya's transverse link invariant in Khovanov homology. 
\end{abstract}

\section{Introduction}
Knot Floer homology is a powerful  invariant  of knots defined by Ozsv{\'a}th and Szab{\'o} \cite{OS3} and by Rasmussen \cite{Ra}. The most basic version  of this invariant assigns to a knot $K$ in a closed, oriented 3--manifold $Y$ a vector space \[\HFKh(Y,K)\] over the field $\F=\Z/2\Z$. If $K$ is nullhomologous with Seifert surface $\Sigma$ then this vector space may be endowed with an  Alexander grading, \[\HFKh(Y,K)=\bigoplus_{i=-g(\Sigma)}^{g(\Sigma)}\HFKh(Y,K,[\Sigma],i), \]  which depends only the relative homology class of the surface in $ H_2(Y,K).$ We will  omit this class from the notation when it is unambiguous, as when $Y$ is a rational homology 3--sphere. This Alexander grading is symmetric in the sense that \[\HFKh(Y,K,[\Sigma],i)\cong\HFKh(Y,K,[\Sigma],-i) \text{ for all i}.\] Much of the power of knot Floer homology owes to its relationship with the genus  and fiberedness of knots. For instance, Ozsv{\'a}th and Szab{\'o} proved \cite{OS4} that if $K\subset Y$ is fibered with fiber $\Sigma$ then \[\HFKh(Y,K,[\Sigma],g(\Sigma))\cong \F.\] Moreover, knot Floer homology completely detects genus \cite{OS-genus} and fiberedness \cite{ghiggini-fibered,yi-fibered} for knots in $S^3$. Specifically, 
\begin{align}
\label{eqn:genus}&\HFKh(S^3,K,g(K))\neq 0\textrm{ and } \HFKh(S^3,K,i)=0 \textrm{ for } i>g(K),\\
\label{eqn:fibered}&\HFKh(S^3,K,g(K))\cong \F \textrm{ if and only if } K \textrm{ is fibered}.
\end{align} 
These facts can be used to prove that knots in $S^3$ with lens space surgeries are fibered \cite{OS-lens}, and that the trefoil and figure eight knots are characterized by   their Dehn surgeries \cite{OS-dehn}.

The results above are all in some way concerned with the summand of knot Floer homology in the top Alexander grading. Much less is known about the summands in other  gradings. Our main result is that the knot Floer homology of a fibered knot is  nontrivial in its next-to-top Alexander grading.

\begin{theorem}
\label{thm:main}
Suppose $K\subset Y$ is a genus $g>0$ fibered knot with fiber  $\Sigma$. Then the summand $\HFKh(Y,K,[\Sigma],g-1)$ is nonzero.
\end{theorem}

While this result may seem a bit technical, we will see below that it can immediately be put to several interesting ends. We  remark that an analogue of Theorem~\ref{thm:main} was recently proven by Baldwin and Sivek \cite{bs-trefoil} in the setting of instanton knot Floer homology, and was used in combination with work of Kronheimer and Mrowka \cite{km-khovanov} to prove that Khovanov homology detects the trefoils. Baldwin and Sivek's proof is based in part on our proof of Theorem~\ref{thm:main}.

\begin{remark}
Although our discussion will focus exclusively on knots, we remark that the conclusion of Theorem~\ref{thm:main} holds for fibered links as well, via the process of \emph{knotiffication}. Indeed, if $K \subset Y$ is a fibered link, then the fibration naturally extends to one of its knotiffication $\kappa(K) \subset \kappa(Y)$. For more information on this construction, we refer the reader to \cite{OS3}.
\end{remark}

\subsection{Applications}

Suppose $Y$ is a rational homology 3--sphere. The  Heegaard Floer homology of $Y$ is bounded in rank by the number of elements in the first homology of $Y$, \[\HFh(Y)\geq |H_1(Y)|.\] An \emph{$L$-space} is a rational homology 3--sphere $Y$ for which  this inequality is sharp, \[\HFh(Y)= |H_1(Y)|.\]  $L$-spaces include lens spaces and all other 3--manifolds with elliptic geometry. A great deal of effort has been devoted to understanding how $L$-spaces  arise via  surgery. A knot $K\subset S^3$ is called an \emph{$L$-space knot} if some Dehn surgery on $K$  is an $L$-space.  Ozsv{\'a}th and Szab{\'o} proved in \cite{OS-lens} that knot Floer homology provides a strong obstruction to having an $L$-space surgery. Namely, if $K$ is an $L$-space knot then \begin{equation}\label{eqn:constraint}\rk\HFKh(S^3,K,i)= 0\textrm{ or } 1\end{equation} for each  Alexander grading $i$. Combined with the genus detection  \eqref{eqn:genus} and fiberedness detection \eqref{eqn:fibered}, this implies in particular that $L$-space knots are fibered. 

 Theorem~\ref{thm:main} therefore imposes the following additional constraint, first proven by Hedden and Watson by quite different means in \cite[Corollary 9]{HW}.\footnote{Hedden and Watson credit Rasmussen as the first to observe this.}

\begin{corollary}
If $K$ is an $L$-space knot then $\HFKh(S^3,K,g(K)-1)\cong \F$.
\end{corollary}

Theorem~\ref{thm:main} also enables a new and immediate proof of the following result, which was first established by Krcatovich using his \emph{reduced} knot Floer complex in \cite[Theorem 1.2]{Kr}.

\begin{corollary}
\label{cor:prime}
$L$-space knots are prime.
\end{corollary}

Theorem~\ref{thm:main}  provides for a simple proof of the following as well, which was first established by Hedden and Watson in \cite[Corollary 8]{HW} using Rasmussen's $h$-invariants.

\begin{corollary}
\label{cor:trefoil} $\rk\HFKh(S^3,K) = 3$ iff $K$ is a trefoil.
\end{corollary}

In fact, we  are able to prove the following more general result by a combination of  Theorem~\ref{thm:main} with the results of \cite{Ba2}.

\begin{corollary}
\label{cor:general}
Suppose $K\subset Y$ is a nullhomologous knot with irreducible complement. Then $\rk\HFKh(Y,K)=3$ iff $K$ is one of the following:
\begin{itemize}
\item a trefoil in $S^3$,
\item the core of $(+1)$-surgery on the right-handed trefoil, 
\item the core of $(-1)$-surgery on the left-handed trefoil.
\end{itemize}
\end{corollary}

\subsection{Antecedents} As mentioned in \cite{HW}, Rasmussen was the first to observe that $L$-space knots have nontrivial knot Floer homology in the next-to-top Alexander grading. Hedden and Watson proved a result akin to Theorem~\ref{thm:main} for knots in $S^3$ under some additional  assumptions on the $\tau$ invariant and the knot Floer homology in the top Alexander grading:

\begin{theorem}[{\cite[Theorem 7]{HW}}]
\label{thm:hw}
Suppose $K\subset S^3$ is a knot of genus $g>1$. If 
\[\tau(K)=g \textrm{ and }
\HFKh_{-1}(S^3,K,g)=0\]
then $\HFKh_{-1}(S^3,K,g-1)\neq 0$. Here, the subscripts denote the Maslov grading.
\end{theorem}

As alluded to above, Hedden and Watson used this result in \cite{HW} to prove that the rank of knot Floer homology detects the trefoil. Though it was not observed in \cite{HW}, their result is also strong enough to show that $L$-space knots are prime, by an argument similar to ours. We emphasize that Hedden and Watson's proof of  Theorem~\ref{thm:hw} is conceptually very different from our proof of Theorem~\ref{thm:main}.

\subsection{Refining  the contact invariant}
The contact invariant in Heegaard Floer homology, defined by Ozsv{\'a}th and Szab{\'o} in \cite{OS4}, assigns to a contact structure $\xi$ on $Y$ a class \[c(\xi)\in\HFh(-Y).\] This class vanishes when $\xi$ is overtwisted. Thus, to prove that a contact structure $\xi$ is tight, it suffices to show that $c(\xi)\neq 0$. This basic principle enabled the classification of Seifert fibered spaces which admit tight contact structures, for instance \cite{LS}. The contact invariant  does not completely detect tightness, however. For example, it vanishes for tight contact manifolds with positive Giroux torsion \cite{GHV}. This paper began as an attempt to develop a refinement of the contact invariant which can obstruct overtwistedness even when the invariant vanishes. We describe our approach below, but do not develop it further in this paper.

Suppose $K$ is the connected binding of an open book $\ob$ compatible with $(Y,\xi)$. The knot $-K\subset -Y$ gives rise to a filtration of the Heegaard Floer complex of $-Y$ which, up to filtered chain homotopy equivalence, takes the form \[\F\langle\mathbf{c}\rangle=\mathscr{F}_{-g}\subset\mathscr{F}_{1-g}\subset \dots\subset \mathscr{F}_g = \CFh(-Y).\] The contact invariant $c(\xi)$ is defined in \cite{OS4} as \[c(\xi):=[\mathbf{c}]\in\mathit{H}_*(\CFh(-Y),\partial) = \HFh(-Y).\] Therefore, if $c(\xi)=0$ then the class $[\mathbf{c}]$ vanishes in the homology of some filtration level. We assign a number $b(\ob)\in\N\cup\{\infty\}$ to this open book which records where this class vanishes, \[b(\ob) :=
\begin{cases}
\infty,& c(\xi)\neq 0,\\
g+\min\{k\mid[\mathbf{c}] = 0\textrm{ in } \mathit{H}_*(\mathscr{F}_k)\},&c(\xi)\neq 0.
\end{cases}\]
One may then  define an invariant of $\xi$ by minimizing over compatible open books, \[b(\xi) = \min\{b(\ob)\mid \ob \textrm{ compatible with } (Y,\xi)\}.\] We prove the following.

\begin{theorem}
\label{thm:b}If $\ob$ is not right-veering then $b(\ob)=1$.
\end{theorem}  This fact is  what inspired and led to the proof of Theorem~\ref{thm:main}. It also  implies the following since every overtwisted contact manifold has a supporting open book with connected binding which is not right-veering.
\begin{corollary}
\label{cor:b}
If $(Y,\xi)$ is overtwisted then $b(\xi)=1$.
\end{corollary}
Thus, to prove that $\xi$ is tight it suffices to show that $b(\xi)>1$.

Theorem~\ref{thm:b} also provides a simpler solution to the word problem in the mapping class group of a surface with connected boundary than in \cite{HW}. Indeed, suppose $\varphi$ is a diffeomorphism of $\Sigma$ which restricts to the identity on $\partial \Sigma$. Then $\varphi = \id$ if and only if both $\varphi$ and $\varphi^{-1}$ are right-veering, which leads to the following.

\begin{corollary}
 $\varphi = \id$ if and only if $b(\Sigma,\varphi)>1<b(\Sigma,\varphi^{-1})$.
\end{corollary}

This invariant motivated the definition of an analogous refinement of Plamenevskaya's transverse invariant in Khovanov homology by  Hubbard and Saltz \cite{hubbard-saltz}. The number $b(\xi)$ is not \emph{a priori} a very calculable contact invariant; nevertheless, there are several interesting questions its construction raises, as described below.

\subsection{Questions}

It is thought  that the knot Floer homology of a fibered knot  in the next-to-top Alexander grading should be related to the symplectic Floer homology of the monodromy of the fibration. This raises the following.

\begin{question}
What implications does Theorem~\ref{thm:main} have for the symplectic Floer homology of mapping classes of a compact surface with connected boundary?
\end{question}

Motivated by  Hedden and Watson's Theorem~\ref{thm:hw}, we ask whether a version of Theorem~\ref{thm:main} holds for arbitrary (nonfibered) knots.

\begin{question}
Is the knot Floer homology of \emph{every}  knot in $S^3$  of positive genus nontrivial in its next-to-top Alexander grading?
\end{question}

Or even more generally, as Sivek has asked:

\begin{question}
Is it true for arbitrary knots $K\subset S^3$ of positive genus that \[\rk\HFKh(S^3,K,g(K)-1) \geq \rk\HFKh(S^3,K,g(K))?\]
\end{question}

Perhaps our Heegaard-diagrammatic proof of Theorem~\ref{thm:main} for fibered knots could be adapted to the setting of \emph{broken} fibrations to  answer  these questions.

Below are some questions about the $b(\ob)$ and $b(\xi)$.
\begin{question}
How does $b(\ob)$ behave under positive stabilization?
\end{question}

We  suspect that $b(\ob)$ is not invariant under positive stabilization. On the other hand, we believe that it is nondecreasing under positive stabilization and can increase by at most $1$. If true, then given  open books $\ob$ and $\ob'$ supporting the same contact structure, the difference $|b(\ob)-b(\ob')|$ provides a lower bound on the total number of positive stabilizations required to achieve a common stabilization.

\begin{question}
Are there contact structures with $\infty > b(\xi)>1$?
\end{question}

\begin{question}
Is $b(\xi)$ nondecreasing under Legendrian surgery?
\end{question}

A similar (but ultimately calculable) numerical refinement of the Heegaard Floer contact invariant was  defined by Kutluhan, Mati{\'c}, Van Horn-Morris, and Wand  in \cite{kmvw}, inspired by work of Hutchings in embedded contact homology \cite[Appendix]{LW}.

\begin{question}
Is there a relationship between $b(\xi)$ and the  invariant defined in \cite{kmvw}?
\end{question}

As mentioned above, $b(\ob)=1$ if $\ob$ is not right-veering. We ask the following.
\begin{question}
Is it the case that $b(\ob)=1$ if and only if $\ob$ is not right-veering?
\end{question} 

Ito and Kawamuro proved in \cite{ito-kawamuro} that $\ob$ is not right-veering if and only if there is a transverse overtwisted disk in the contact manifold  $(Y,\xi)$ supported by $\ob$ such that the open book foliation on this disk has exactly one negative elliptic singularity (where the disk intersects the binding). Their proof suggests the following interesting question.

\begin{question}
Does $b(\ob)$ provide a lower bound for the number of negative elliptic singularities of the open book foliation on any transverse overtwisted disk in $(Y,\xi)$?
\end{question}

\subsection{Organization}
We prove Theorem~\ref{thm:main} and Theorem~\ref{thm:b} in Section~\ref{sec:proof}. Section~\ref{sec:cor} contains  the proofs of Corollaries \ref{cor:prime}, \ref{cor:trefoil}, and \ref{cor:general}.

\subsection{Acknowledgements} We thank John Etnyre, Matt Hedden, Sucharit Sarkar, Jeremy van Horn-Morris, Steven Sivek, and Andy Wand for helpful conversations. We also thank  ICERM for hosting us during the 2014 workshop on \emph{Combinatorial Link Homology Theories, Braids, and Contact Geometry}. Most of this work  was carried out during that workshop.

\section{The proof of Theorem~\ref{thm:main}}
\label{sec:proof}

In this section, we give what we feel is a conceptually simple proof of Theorem~\ref{thm:main},  based on the notion of non-right-veering monodromy and the fact that the boundary map in the Heegaard Floer complex squares to zero. We  will assume the reader is   familiar with Heegaard Floer homology and  with Honda, Kazez, and Mati{\'c}'s description of the Ozsv{\'a}th-Szab{\'o} contact invariant in \cite{HKM1}, though we provide a  cursory review below, in part for completeness but largely in order to establish notation and terminology. 
\subsection{Heegaard Floer homology}
To define the Heegaard Floer homology of a closed, oriented 3--manifold $Y$  one starts with a \emph{pointed Heegaard diagram} \[\mathcal{H} = (S,\bs\alpha,\bs\beta,w)\] for $Y$. The  Heegaard Floer chain complex $\CFh(\mathcal{H})$ is the  $\F$-vector space generated by intersection points between the associated tori $\T_{\alpha}$ and $\T_{\beta}$ in $\Sym^g(S)$. The differential \[\partial:\CFh(\mathcal{H})\to\CFh(\mathcal{H})\] is the linear map defined on generators by \[\partial(\mathbf{x}) = \sum_{\mathbf{y}\in\T_{\alpha}\cap\T_{\beta}}\sum_{\substack{\phi\in\pi_2(\mathbf{x},\mathbf{y})\\\mu(\phi)=1\\n_w(\phi)=0}} \#\widehat{\mathcal{M}}(\phi)\cdot \mathbf{y},\] where $\pi_2(\mathbf{x},\mathbf{y})$ is the set of homotopy classes of Whitney disks from $\mathbf{x}$ to $\mathbf{y}$; $\mu(\phi)$ is  the Maslov index of $\phi$; $n_w(\phi)$ is the intersection number \[\phi\cdot (\{w\}\times\Sym^{g-1}(S));\] and $\widehat{\mathcal{M}}(\phi)$ is the space of pseudo-holomorphic representatives of $\phi$ modulo conformal automorphisms of the domain. The Heegaard Floer homology of $Y$ is the homology of this  complex, \[\HFh(Y)=H_*(\CFh(\mathcal{H}),\partial),\] and is an invariant of $Y$. 

\subsection{Knot Floer homology} 
\label{sec:hfk} To define the knot Floer homology of a null-homologous knot $K$ in a closed, oriented 3--manifold  $Y$, one starts with a \emph{doubly-pointed} Heegaard diagram \[\mathcal{H} = (S,\bs\alpha,\bs\beta,z,w)\] for $K\subset Y$. In particular,
\begin{itemize}
\item $(S,\bs\alpha,\bs\beta,w)$ is a pointed Heegaard diagram for $Y$, and
\item if $\gamma_{\alpha}\subset S\ssm\bs\alpha$ is an arc from $z$ to $w$ and  $\gamma_{\beta}\subset S\ssm\bs\beta$ is an arc from $w$ to $z$ then $K$ is the union of the arcs obtained by pushing the interiors of $\gamma_{\alpha}$ and $\gamma_{\beta}$ into the $\alpha$- and $\beta$-handlebodies, respectively.
\end{itemize}
Given a Seifert surface $\Sigma$ for $K$, each generator $\mathbf{x}$ of the Heegaard Floer complex \[\CFh(\mathcal{H}):=\CFh(S,\bs\alpha,\bs\beta,w)\] is assigned an \emph{Alexander grading} \[A_{[\Sigma]}(\mathbf{x}) = \frac{1}{2}\langle c_1(\mathfrak{s}_{\mathbf{x}}), [\Sigma] \rangle \in \Z\] which depends only on the relative homology class of $\Sigma$. For generators $\mathbf{x}$ and $\mathbf{y}$ connected by a Whitney disk $\phi\in\pi_2(\mathbf{x},\mathbf{y})$, the relative Alexander grading is given by \[A_{[\Sigma]}(\mathbf{x})-A_{[\Sigma]}(\mathbf{y})=n_z(\phi)-n_w(\phi).\] 
Let $\mathcal{F}_i$ denote the subspace of $\CFh(\mathcal{H})$ spanned by generators with Alexander grading at most $i$. That $\mathcal{F}_i$ is a subcomplex  follows  the positivity of $n_z(\phi)$ for $\phi$ with  pseudo-holomorphic  representatives and that fact that $\partial$ counts disks with $n_w(\phi)=0$. These subcomplexes  define a filtration \[\cdots \mathscr{F}_i\subset\mathscr{F}_{i+1}\subset\cdots\subset\mathscr{F}_j=\CFh(\mathcal{H})\] whose filtered chain homotopy type is an invariant of $(Y,K,[\Sigma])$. 

The knot Floer chain complex $\CFKh(\mathcal{H})$ is   the associated graded object of this filtration. Equivalently, it is the complex generated by intersection points in $\T_{\alpha}\cap \T_{\beta}$ who differential counts disks as before which satisfy the extra condition that  $n_z(\phi)=0$. The knot Floer homology  of $K$ in Alexander grading $i$ is given by \[\HFKh(Y,K,[\Sigma],i)=H_*(\mathscr{F}_i/\mathscr{F}_{i-1}).\] These groups vanish for $|i|>g(\Sigma)$. Moreover, the Alexander grading is symmetric and detects genus and fiberedness as described in the introduction. It is customary to define the knot Floer homology of $K$ to be the graded group \[\HFKh(Y,K)=\HFKh(\mathcal{H}) := \mathit{H}_*(\CFKh(\mathcal{H}))=\bigoplus_{i=-g(\Sigma)}^{g(\Sigma)}\HFKh(Y,K,[\Sigma],i).\]

We describe below another useful way of computing the relative Alexander grading. Let $\gamma = \gamma_{\alpha}\cup\gamma_{\beta}$ be the union of the arcs defined above. Let $\mathcal{D}_1,\dots,\mathcal{D}_k$ denote the closures of the components of $\Sigma\ssm(\bs\alpha\cup\bs\beta\cup\gamma)$. 
\begin{definition} A \emph{relative periodic domain} is a 2--chain $\mathcal{P} = \sum a_i\mathcal{D}_i$ satisfying \[\partial \mathcal{P} = \ell\gamma + \sum n_i\alpha_i + \sum m_i\beta_i\] for integers $\ell,n_i,m_i$. 
\end{definition} Hedden and Plamenevskaya proved in \cite[Lemma 2.3]{HP} that if $\mathcal{P}$ is a relative periodic domain representing the homology class $[\Sigma]$ then the relative Alexander grading between generators $\mathbf{x},\mathbf{y}\in\T_{\alpha}\cap\T_{\beta}$ is given by \begin{equation}\label{eqn:alex}A_{[\Sigma]}(\mathbf{x})-A_{[\Sigma]}(\mathbf{y}) = n_{\mathbf{x}}(\mathcal{P})-n_{\mathbf{y}}(\mathcal{P}).\end{equation}

\subsection{The contact invariant}
Suppose $(\Sigma,\varphi)$ is an open book decomposition for $Y$. A \emph{basis} for $\Sigma$ is a collection $\{a_1,\dots,a_{2g}\}$ of disjoint, properly embedded arcs in $\Sigma$ whose complement  is a disk. Given such a basis, let $b_i$ be an isotopic copy of $a_i$ obtained by shifting the endpoints of $a_i$ along $\partial\Sigma$ in the direction specified by its orientation, so that $b_i$  intersects $a_i$ transversally in a single point $c_i$, as  shown in Figure~\ref{fig:openbook} in the case $g=1$. Following Honda, Kazez, and Mati{\'c} \cite{HKM1}, we  form  a pointed Heegaard diagram \[(S,\bs\alpha = \{\alpha_1,\dots,\alpha_{2g}\},\bs\beta = \{\beta_1,\dots,\beta_{2g}\},w)\] for $Y$ \emph{adapted to this open book and basis} by ``doubling" the fiber surface and  basis arcs. More precisely: 
\begin{itemize}
\item $S=\Sigma\cup-\Sigma$ is the union of two copies of $\Sigma$ glued along their boundaries,
\item $\alpha_i = a_i\cup a_i$,
\item $\beta_i = b_i \cup \varphi(b_i)$,
\item $w\in \Sigma\subset S$ lies outside of the thin   regions traced   by the isotopies from the $a_i$ to $b_i$,
\end{itemize}
as illustrated  in Figure~\ref{fig:openbook}. 

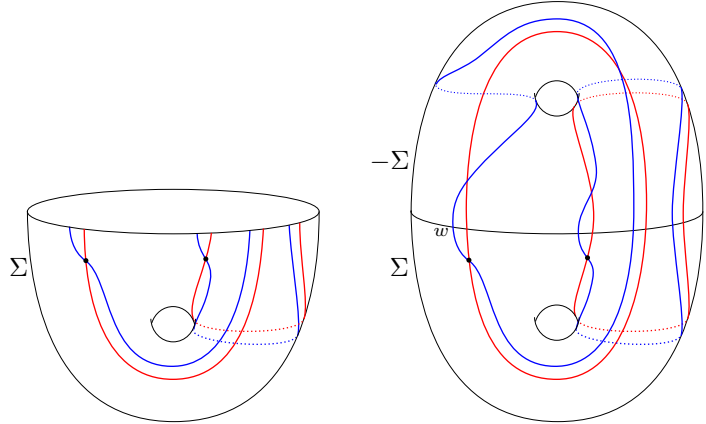
\begin{figure}[ht]
\labellist
\small \hair 2pt

\pinlabel $\Sigma$ at -8 160
\pinlabel $\Sigma$ at 382 160
\pinlabel $-\Sigma$ at 372 268

\tiny
\pinlabel $w$ at 425 195

\endlabellist
\centering
\includegraphics[width=9cm]{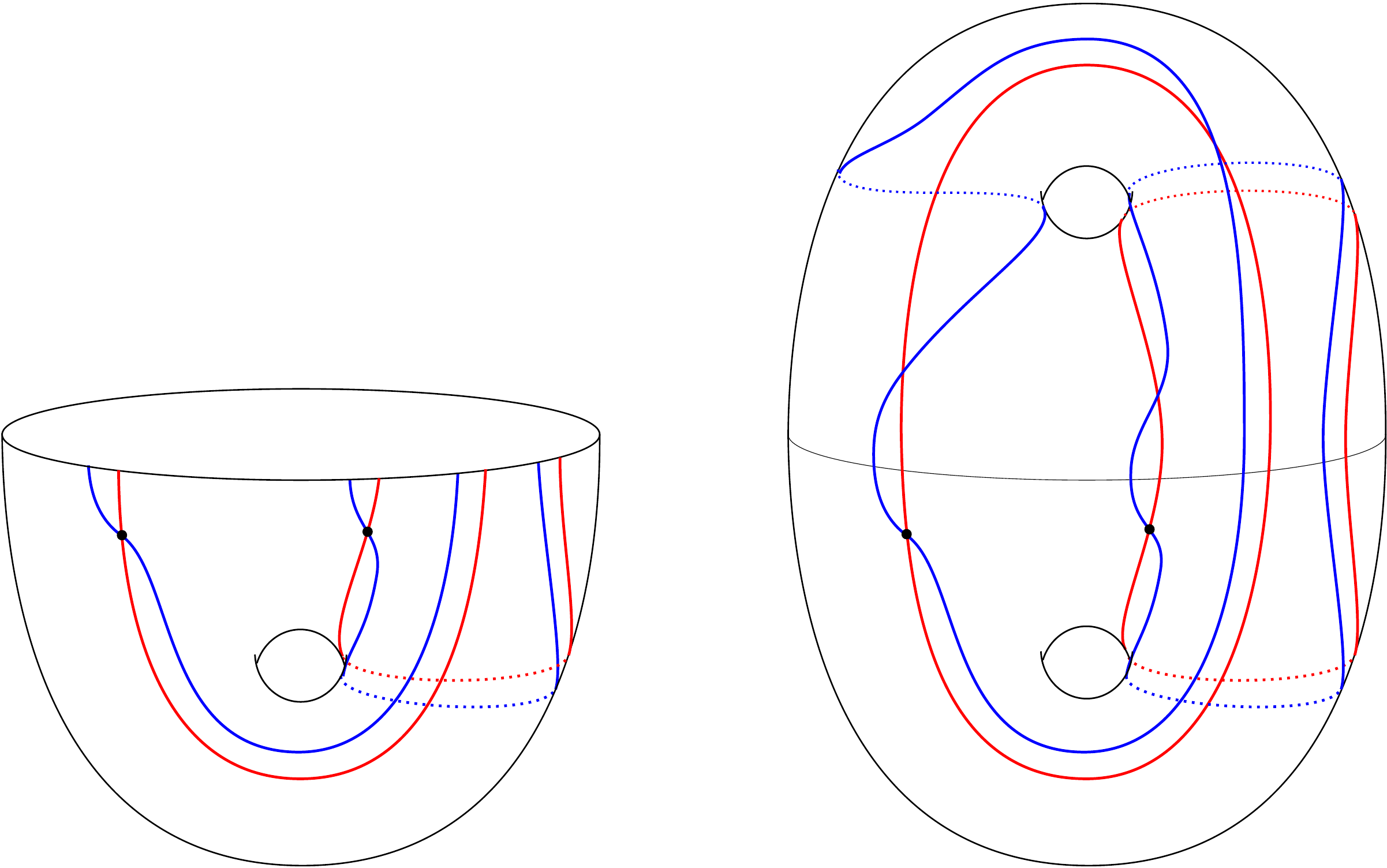}
\caption{Left, the arcs $a_1,a_2$ in red and $b_1,b_2$ in blue. Right, the corresponding  pointed Heegaard diagram with $\alpha_1, \alpha_2$  in red and $\beta_1,\beta_2$  in blue, where $\varphi$ is a left-handed Dehn twist around a curve in $\Sigma$ dual to $a_1$ (which looks like a right-handed Dehn twist on $-\Sigma$). The intersection points $c_1$ and $c_2$ are shown as black dots.
}
\label{fig:openbook}
\end{figure}

Then $\mathcal{H} = (S,\bs\beta,\bs\alpha,w)$ is a pointed Heegaard diagram for $-Y$. Note that the \emph{contact generator} \[\mathbf{c}:=\{c_1,\dots,c_{2g}\}\] is a cycle in $\CFh(\mathcal{H})$.  Moreover, Honda, Kazez, and Mati{\'c} proved in  \cite{HKM1}  that this cycle represents the contact invariant of $\xi$, \[c(\xi)=c(\Sigma,\varphi)=[\mathbf{c}]\in \HFh(-Y),\] originally defined by Ozsv{\'a}th and Szab{\'o} in \cite{OS4}.

\subsection{The proof of Theorem~\ref{thm:main}}
Let us assume henceforth that $K$ is a genus $g>0$ fibered knot in $Y$ with fiber $\Sigma$, as in the hypothesis of Theorem~\ref{thm:main}. Let $(\Sigma,\varphi)$ be an open book corresponding to the fibration of $K$, supporting a contact structure $\xi$ on $Y$. 
Let $(S,\bs\alpha,\bs\beta,w)$ be a pointed Heegaard diagram for $Y$ adapted to this open book and  a basis $\{a_1,\dots,a_{2g}\}$, as described in the previous section.

To turn this diagram into a doubly-pointed Heegaard diagram for $K\subset Y$, we perform finger moves on the $\bs\beta$ curves in a neighborhood of $\partial \Sigma\subset S$, pushing these curves in the direction of the orientation of $\partial \Sigma$, and place  a basepoint $z$  in a region of  $S\ssm (\bs\alpha\cup \bs\beta)$ adjacent to $a_1$, so that  $\partial\Sigma$ is given as the union of an arc $\gamma_\alpha\subset S\ssm\bs\alpha$ from $z$ to $w$ with an arc $\gamma_\beta\subset S\ssm\bs\beta$ from $w$ to $z$, as shown in Figure~\ref{fig:finger}. 

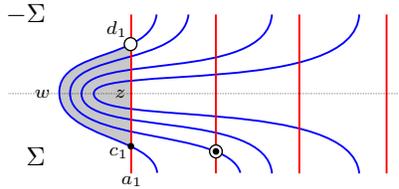
\begin{figure}[ht]
\labellist
\small \hair 2pt

\pinlabel $\Sigma$ at 17 11
\pinlabel $-\Sigma$ at 11 97

\tiny
\pinlabel $w$ at 21 49
\pinlabel $z$ at 68 49
\pinlabel $a_1$ at 75 -4
\pinlabel $c_1$ at 67 13

\pinlabel $d_1$ at 66 88

\endlabellist
\centering
\includegraphics[width=5.3cm]{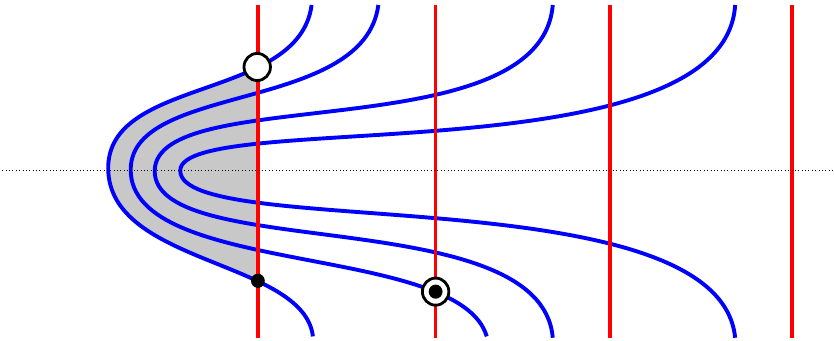}
\caption{Modifying the Heegaard diagram in Figure~\ref{fig:openbook}. The black intersection points comprise $\mathbf{c}$ and the white intersection points comprise $\mathbf{d}$. The bigon from $\mathbf{d}$ to $\mathbf{c}$ is shown in gray.
}
\label{fig:finger}
\end{figure}

We will be most interested in the diagram \[\mathcal{H} = (S,\bs\beta,\bs\alpha,z,w),\] obtained  by reversing the roles of $\bs\alpha$ and $\bs\beta$. This is  a doubly-pointed diagram for   $-K\subset -Y$. We will study the knot Floer homology of $-K$ with respect to the Alexander filtration induced by its Seifert surface $-\Sigma$. Note that the intersection points $c_1,\dots,c_{2g}$ naturally persist in this diagram. Moreover, the contact generator \[\mathbf{c}:=\{c_1,\dots,c_{2g}\}\] is a cycle in the complexes \[\CFKh(\mathcal{H}) \textrm{ and }\CFh(\mathcal{H}):=\CFh(S,\bs\beta,\bs\alpha,w),\] and  still represents the contact invariant, \[c(\xi)=[\mathbf{c}]\in\HFh(-Y),\] as the finger move procedure results in a Heegaard diagram adapted to an open book $(\Sigma,\varphi')$ where $\varphi'$ is simply the composition of $\varphi$ with an isotopy.

The following lemma characterizes the Alexander gradings of generators of $\CFh(\mathcal{H})$.

 \begin{lemma}
 \label{lem:grading} The Alexander grading of a generator $\mathbf{x}$ of $\CFh(\mathcal{H})$ is given by \[A_{[-\Sigma]}(\mathbf{x}) = \frac{1}{2}\cdot(n_{\mathbf{x}}(-\Sigma)-n_{\mathbf{x}}(\Sigma)).\] Equivalently, $A_{[-\Sigma]}(\mathbf{x})$ is equal to the number of components of $\mathbf{x}$ in $-\Sigma\subset S$ minus $g$.
 \end{lemma}
 
 \begin{proof}
 Note that the region $-\Sigma\subset S$ is a relative periodic domain representing $[-\Sigma]$. The formula \eqref{eqn:alex} for relative Alexander grading implies that the largest Alexander grading $\bar A$ is realized precisely by generators with components   contained entirely in $-\Sigma\subset S$ while the smallest Alexander grading $\underline{A}$ is realized precisely by generators with components contained in $\Sigma\subset S$. In addition, the difference \[\bar A-\underline{A}=2g.\]  The fact that  the knot Floer homology of $-K$ with respect to $[-\Sigma]$ is  nontrivial in gradings $\pm g$ then forces $\bar A= g$. The lemma now follows easily from the formula \eqref{eqn:alex}. 
 \end{proof}

\begin{remark}
\label{rmk:filt}Lemma~\ref{lem:grading} implies that the filtration of $\CFh(\mathcal{H})$ induced by $-K$ and $-\Sigma$ takes the form \[\emptyset = \mathscr{F}_{-g-1}\subset \mathscr{F}_{-g}\subset  \cdots \subset \mathscr{F}_g=\CFh(\mathcal{H}),\] and that the cycle $\mathbf{c}$ is contained in the bottom filtration level $\mathscr{F}_{-g}$. 
\end{remark}

We  claim that the cycle $\mathbf{c}$  represents a nonzero class in knot Floer homology:

\begin{theorem}
\label{thm:nonzero} The class \[[\mathbf{c}]\in\HFKh(-Y,-K,[-\Sigma],-g) = \mathit{H}_*(\mathscr{F}_{-g}/\mathscr{F}_{-g-1})=\mathit{H}_*(\mathscr{F}_{-g})\cong \F\] is always nonzero.
\end{theorem}

We  postpone the proof of Theorem~\ref{thm:nonzero}  until the end of this section. We first show how it  proves Theorems \ref{thm:main} and \ref{thm:b}, beginning with a definition due to Honda, Kazez, and Mati{\'c} \cite{HKM4} and some  remarks.

\begin{definition}
 Given a properly embedded arc $a\subset \Sigma$, we say that $\varphi$ \emph{sends $a$ to the left at an endpoint $p$} if $\varphi(a)$ is not isotopic to $a$ and if, after isotoping $\varphi(a)$ so that it intersects $a$ minimally, $\varphi(a)$ is to the left of $a$ near $p$, as shown in Figure~\ref{fig:left2}. The map $\varphi$ is   \emph{right-veering} if it does not send any  arc to the left at one of its endpoints.
 \end{definition}

\begin{figure}[ht]
\labellist
\small \hair 2pt

\pinlabel $a$ at 47 26
\pinlabel $\varphi(a)$ at 17 26
\pinlabel $p$ at 39 -5
\pinlabel $\Sigma$ at 65 57

\endlabellist
\centering
\includegraphics[width=2cm]{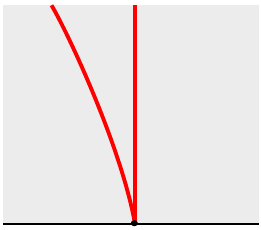}
\caption{$\varphi$ sends $a$ to the left at $p$.
}
\label{fig:left2}
\end{figure}
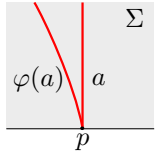

\begin{remark}
\label{rmk:nrv}For the proof of Theorem~\ref{thm:main}, we may assume without loss of generality that the monodromy $\varphi$ is not right-veering. 
Indeed, suppose first that $\varphi=\id$. Then Theorem~\ref{thm:main} holds by a calculation in \cite[Section 9]{OS3}. If $\varphi\neq \id$ then one of $\varphi$ or $\varphi^{-1}$  is  not right-veering. If $\varphi$ is right-veering then we  use the fact that knot Floer homology is invariant under reversing the orientation of $Y$ and consider instead the knot $K\subset -Y$ with open book $(\Sigma,\varphi^{-1})$. 
\end{remark}
 
\begin{remark}
\label{rmk:symmetry}Since Alexander graded knot Floer homology is symmetric and invariant under reversing the orientation of $Y, K, \Sigma$, \[\HFKh(Y,K,[\Sigma],i) = \HFKh(-Y,-K,[-\Sigma],-i),\]
it suffices for Theorem~\ref{thm:main} to prove that \[\HFKh(-Y,-K,[-\Sigma],1-g)\neq 0.\] \end{remark}

\begin{proof}[Proof of Theorem~\ref{thm:main}]
We will prove that \[\HFKh(-Y,-K,[-\Sigma],1-g)\neq 0\] per Remark~\ref{rmk:symmetry}. We will  also assume that $\varphi$ is not right-veering per Remark~\ref{rmk:nrv}. In \cite[Proof of Lemma 3.2]{HKM1}, Honda, Kazez, and Mati{\'c} show that the fact that $\varphi$ is not right-veering means that there is some \emph{nonseparating} arc $a_1\subset \Sigma$ which is sent to the left by $\varphi$ at one of its endpoints. Since $a_1$ is nonseparating, it can be completed to a basis $\{a_1,\dots,a_{2g}\}$ for $\Sigma$. Let $\mathcal{H} = (S,\bs\beta,\bs\alpha,z,w)$ be the doubly-pointed Heegaard diagram for $-K\subset -Y$ associated to this basis as above.

It is then easy to see that the non-right-veering-ness implies that the curves $\alpha_1$ and $\beta_1$ form a single bigon with corners at intersection points $d_1, c_1\subset \alpha_1\cap\beta_1$, after isotoping these curves on $-\Sigma\subset S$ to intersect minimally, as shown in Figure~\ref{fig:finger}. Consider the generator \[\mathbf{d}=\{d_1,c_2,\dots,c_{2g}\}.\] By Lemma~\ref{lem:grading}, we have that \[A_{[-\Sigma]}(\mathbf{d}) = 1-g.\] Moreover, it is easy to see that the bigon above is the sole contribution to $\partial \mathbf{d}$ by the same reasoning which shows that $\mathbf{c}$ is a cycle (informally, no holomorphic disk avoiding $w$ can have corners at $c_2,\dots,c_{2g}$), so that \[\partial \mathbf{d}=\mathbf{c}.\] Thus, $\mathbf{d}$ is  a cycle in $\mathscr{F}_{1-g}/\mathscr{F}_{-g}$. We claim that $\mathbf{d}$ is not a boundary in this quotient. This will imply that the class \[[\mathbf{d}]\in \HFKh(-Y,-K,[-\Sigma],1-g) = \mathit{H}_*(\mathscr{F}_{1-g}/\mathscr{F}_{-g})\] is nonzero, which will prove Theorem~\ref{thm:main}. 

Suppose for a contradiction that $\mathbf{d}$ is  a boundary in $\mathscr{F}_{1-g}/\mathscr{F}_{-g}$. Then there is a homogeneous chain \[\mathbf{e}\in\CFh(\mathcal{H})\] in Alexander grading $1-g$ such that \[\partial \mathbf{e} = \mathbf{d}+\mathbf{f},\] where $\mathbf{f}$ is a chain in Alexander grading $-g$. But the fact that $\partial\circ\partial=0$ then forces \[\partial(\partial(\mathbf{e}))=\partial \mathbf{d} + \partial \mathbf{f} = \mathbf{c} + \partial \mathbf{f} = 0.\] That is, $\partial \mathbf{f} = \mathbf{c}$. But this contradicts the fact that $\mathbf{c}$ represents a nonzero element in \[\HFKh(-Y,-K,[-\Sigma],-g)=\mathit{H}_*(\mathscr{F}_{-g}),\] as claimed in Theorem~\ref{thm:nonzero}.
\end{proof}

\begin{proof}[Proof of Theorem~\ref{thm:b}]
Let $\ob=(\Sigma,\varphi)$ and suppose $\varphi$ not right-veering. Then $c(\xi)=0$, as follows from the proof of Theorem~\ref{thm:main}  (see \cite[Lemma 3.2]{HKM1} for the original proof). Therefore, per the introduction,  $b(\ob)$ is defined as \[b(\ob)=g+\min\{k\mid [\mathbf{c}]=0 \textrm{ in }\mathit{H}_*(\mathscr{F}_k)\}.\] According to Theorem~\ref{thm:nonzero}, $\mathbf{c}$ represents a nonzero class in $\mathit{H}_*(\mathscr{F}_{-g})$. On the other hand, the proof of Theorem~\ref{thm:main} shows that $[\mathbf{c}]=0$ in $\mathit{H}_*(\mathscr{F}_{1-g})$ since $\mathbf{c}=\partial \mathbf{d}$. Thus, \[b(\ob)=g+1-g=1,\] as claimed.
\end{proof}

\begin{remark}
\label{rmk:idtrivial}
Remark~\ref{rmk:nrv} and the proof of Theorem~\ref{thm:main} show that if $K\subset Y$ is a fibered genus $g>0$ knot with open book $(\Sigma,\varphi)$ and $\varphi$ is not the identity then without loss of generality we can assume that $\varphi$ is not right-veering in which case there is a nontrivial differential \[d_1:\HFKh(-Y,-K,[\Sigma],1-g)\to\HFKh(-Y,-K,[\Sigma],-g)\] in the spectral sequence from $\HFKh(-Y,-K)$ to $\HFKh(-Y)$, sending $[\mathbf{d}]$ to $[\mathbf{c}]$. In particular, \[\rk\HFKh(Y,K) \neq \rk\HFh(Y)\] unless $\varphi=\id$.
\end{remark}

It remains  to prove Theorem~\ref{thm:nonzero}.  As we shall see, this follows from a relatively straightforward lemma. To set the stage for this lemma, recall that any two bases for $\Sigma$ can be obtained from one another by a sequence of \emph{arcslides}, where an arcslide is a modification of a basis in which the foot of one basis arc is slid up and over another basis arc as in Figure~\ref{fig:arcslide} below; see \cite{HKM1} for details. The lemma below asserts that the same holds  even when  we disallow arcslides which pass over a fixed basepoint on $\partial \Sigma$.

\begin{figure}[ht]
\labellist
\tiny \hair 2pt

\pinlabel $a_1$ at 14 128
\pinlabel $a_2$ at 203 128
\pinlabel $\partial \Sigma$ at -11 90
\pinlabel $\Sigma$ at 108 128

\endlabellist
\centering
\includegraphics[width=5.5cm]{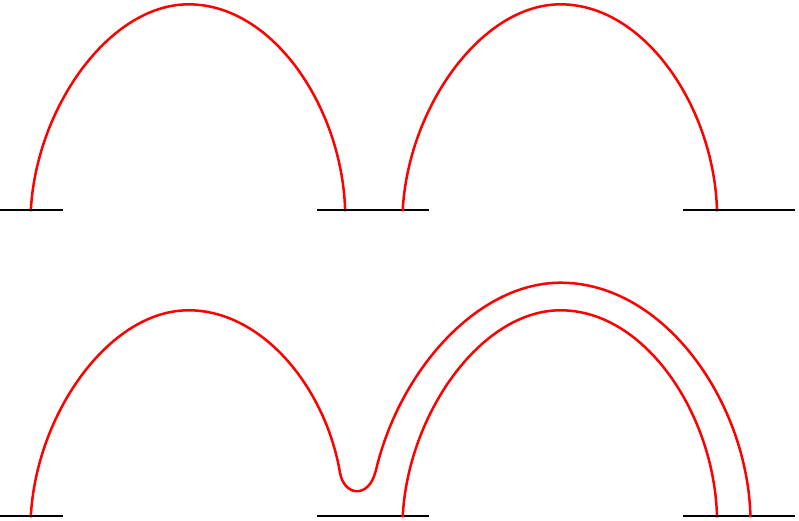}
\caption{Sliding $a_1$ over $a_2$.
}
\label{fig:arcslide}
\end{figure}
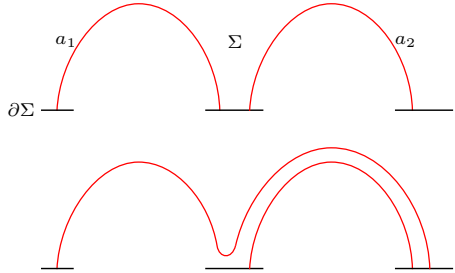

\begin{lemma}
\label{lem:arcslide}
Let $z$ be a basepoint on the boundary of the fiber surface $\Sigma$. Suppose \[\mathbf{a}=\{a_1,\dots,a_{2g}\}\textrm{ and } \mathbf{a}'=\{a'_1,\dots,a'_{2g}\}\] are bases for $\Sigma$  disjoint from $z$. Then $\mathbf{a}'$ can be obtained from $\mathbf{a}$ via  a sequence of arcslides taking place in the complement of $z$. 
\end{lemma}

\begin{proof}
Honda, Kazez and Mati\'c showed in \cite[Proposition~3.4]{HKM1} that the there is a sequence of arcslides taking $\mathbf{a}$ to $\mathbf{a}'$. Thus, to prove Lemma~\ref{lem:arcslide}, it suffices to show that any arcslide which involves sliding a foot over the basepoint $z$ can be effectuated by an alternative sequence of arcslides in the complement of $z$.  In fact, it is sufficient to prove a yet simpler statement, that an isotopy of bases which slides the foot of one arc over $z$ can be accomplished by a sequence of arcslides taking place in the complement of $z$.

For the latter, suppose  $a_1\in \mathbf{a}$ is a basis element with  one foot immediately to one side of $z$, as shown in Figure~\ref{fig:basisslide}. Let $a_1'$ be the isotopic copy of $a_1$ obtained by sliding this foot  over $z$, as  in the the figure. To prove the lemma, it suffices to show that $a_1$ can be arcslid over the other $a_j$ in some sequence until the resulting arc is isotopic to $a_1'$, via arcslides in the complement of $z$. For this, it suffices to show that there is a polygon $P$ in $\Sigma$ with boundary $\partial P$ consisting of arcs of $\partial \Sigma$ together with $a_1$ and $a_1'$ as well as some number of copies of the other $a_j$, such that $z\notin\partial P$. But this is plainly obvious, as cutting $\Sigma$ along $a_1,a_1',a_2,\dots,a_{2g}$ yields the disjoint union of such a polygon $P$ together with a quadrilateral $Q$ bounded by $a_1, a_1'$ and two boundary arcs, with $z\in\partial Q$, as shown in Figure~\ref{fig:basisslide}. In particular, $a_1$ can be transformed into $a_1'$ by sliding each of its endpoints exactly once  over every other basis arc, in the complement of $z$.

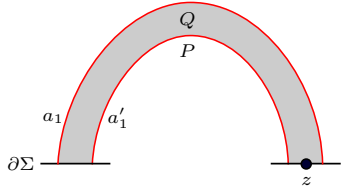
\begin{figure}[ht]
\labellist
\tiny \hair 2pt

\pinlabel $a_1$ at 6 25
\pinlabel $a_1'$ at 36 25
\pinlabel $\partial \Sigma$ at -9 4
\pinlabel $z$ at 123 -4

\pinlabel $Q$ at 68 70
\pinlabel $P$ at 68 56

\endlabellist
\centering
\includegraphics[width=4cm]{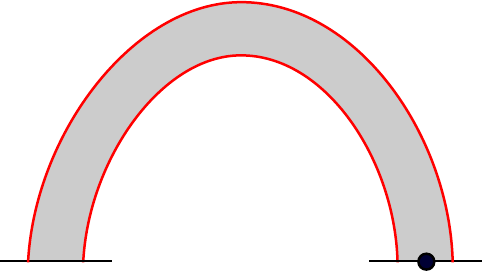}
\caption{The arcs $a_1$ and $a_1'$. The complement   of the quadrilateral $Q$  is a polygon $P\subset \Sigma$ certifying that $a_1$ can be arcslid over the other $a_j$ in the complement of $z$ until the result is isotopic to $a_1'$.
}
\label{fig:basisslide}
\end{figure}

\end{proof}

We may now prove Theorem~\ref{thm:nonzero}.

\begin{proof}[Proof of Theorem~\ref{thm:nonzero}]
We first note that the class $[\mathbf{c}]$ is independent of basis in the following sense. Suppose \[\mathcal{H}=(S,\bs\beta,\bs\alpha,z,w) \textrm{ and }\mathcal{H}'=(S,\bs\beta',\bs\alpha',z,w)\] are  doubly-pointed Heegaard diagrams for $-K\subset -Y$ adapted to the open book $(\Sigma,\varphi)$ and bases $\mathbf{a}$ and $\mathbf{a}'$, where $\mathbf{a}'$ is obtained from $\mathbf{a}$ by a single arcslide in the complement of $z$. Let $\mathbf{c}$ and $\mathbf{c}'$ denote the corresponding contact generators. We claim that there is an isomorphism \[\psi:\HFKh(\mathcal{H})\to\HFKh(\mathcal{H}')\] with \[\psi([\mathbf{c}])=[\mathbf{c}'].\] The proof is identical to that of \cite[Lemma 3.5]{HKM1}. In summary, the diagram $\mathcal{H}'$ is obtained from $\mathcal{H}$ by a pair of handleslides (one for each of  $\bs\alpha$ and $\bs\beta$) in the complement of the basepoints $z$ and $w$. The isomorphism $\psi$ is  induced by the composition of the  holomorphic triangle-counting chain maps associated to these handleslides. A composition of ``small" triangles certifies that the  composite chain map sends the generator $\mathbf{c}$ to $\mathbf{c}'$.

We next note that the class $[\mathbf{c}]$ is natural with respect to the map on knot Floer homology induced by adding  negative Dehn twists to $\varphi$ along nonseparating curves on $\Sigma$. To be precise, suppose  \[\varphi' = \varphi\circ \tau^{-1}_\gamma,\] where $\tau_\gamma$ is a positive Dehn twist around a nonseparating curve $\gamma\subset \Sigma$. Then $(\Sigma, \varphi')$ is an open book decomposition  for the 3--manifold $Y'$  obtained via $(+1)$-surgery on a copy of $\gamma\subset Y$ with respect to the framing  induced by $\Sigma$. Let $K'\subset Y'$ denote the image of $K$ in this surgered manifold, corresponding to the binding of $(\Sigma,\varphi')$. Since $\gamma$ does not link $K$ in $Y$, the cobordism from $Y$ to $Y'$ corresponding to this surgery induces a map \[f:\HFKh(-Y,-K)\to\HFKh(-Y',-K')\] which preserves the Alexander gradings associated to  $-\Sigma$ \cite[Proposition 8.1]{OS3}.  We make two claims about this map below, after reminding the reader  how it is defined.

Since $\gamma$ is nonseparating, we can find a basis $\{a_1,\dots,a_{2g}\}$ for $\Sigma$ such that $\gamma$ intersects $a_1$  in one point and is disjoint from the other $a_i$. Let \[\mathcal{H} = (S,\bs\beta,\bs\alpha,z,w)\] be the doubly-pointed Heegaard diagram for $-K\subset -Y$ adapted to this basis and open book $(\Sigma,\varphi)$. Then one obtains a doubly-pointed Heegaard diagram  \[\mathcal{H}' = (S,\bs\beta',\bs\alpha,z,w)\] for $-K'\subset -Y'$ adapted to this basis and  open book $(\Sigma,\varphi')$ by modifying $\beta_1$ by the negative Dehn twist $\tau_\gamma^{-1}$. Let $\mathbf{c}$ and $\mathbf{c}'$ denote the corresponding contact generators. Our first claim is  that the cobordism map \[f:\HFKh(\mathcal{H})\to\HFKh(\mathcal{H}')\] satisfies \[f([\mathbf{c}])=[\mathbf{c}'].\] The proof is identical to that of \cite[Proposition 3.7]{HKM1}. In summary, the map $f$ is  induced by a  holomorphic triangle-counting chain map. Due to our choice of basis, there is a ``small" triangle which certifies that this chain map sends  $\mathbf{c}$ to $\mathbf{c}'$.

Our second claim is that the restriction \[f:\HFKh(-Y,-K,[-\Sigma],-g)\to\HFKh(-Y',-K',[-\Sigma],-g)\] of $f$ to the summand in the bottommost Alexander grading is an isomorphism. For this, we note that $f$ fits into a surgery exact triangle \begin{equation*}\label{eqn:surgtri} \xymatrix@C=-55pt@R=30pt{
\HFKh(-Y,-K,[-\Sigma],-g) \ar[rr]^{f} & & \HFKh(-Y',-K',[-\Sigma],-g) \ar[dl] \\
& \HFKh(-Y'',-K'',[-\Sigma],-g) \ar[ul], & \\
} \end{equation*}
where $Y''$ is the result of $0$-surgery on $\gamma\subset Y$ and $K''$ is the induced knot. That $f$ is an isomorphism then follows  from the fact that the third group is zero, as $\Sigma$ is homologous in $Y''$ to a surface of genus $g-1$ obtained by cutting $\Sigma$ open along $\gamma$ and capping the new boundary components with disks. Combining the last two claims, we have that $[\mathbf{c}]$ is nonzero if and only if $[\mathbf{c}']$ is.

Let $\tau_\delta$ denote a positive Dehn twists around a curve $\delta$ parallel to $\partial \Sigma$. Note that any other diffeomorphism $\varphi$ of $\Sigma$ is obtained from $\tau_\delta^n$ by adding negative Dehn twists, for some $n>0$. Thus, to complete the proof of Theorem~\ref{thm:nonzero}, it suffices to show that the class \[[\mathbf{c}]\in\HFKh(-Y,-K,[-\Sigma],-g)\] is nonzero in the case that $\varphi = \tau_\delta^n$ for some $n>0$. But this follows immediately from the fact that the class \[c(\xi) = [\mathbf{c}]\in \HFh(-Y)\] is nonzero as $\xi$  Stein fillable in this case \cite{OS4}. To elaborate, suppose $\mathcal{H}$ is a doubly-pointed Heegaard diagram for $K\subset Y$ adapted to $(\Sigma,\varphi)$ and some basis. Recall from Remark~\ref{rmk:filt} that $\mathbf{c}$ is in the bottommost filtration level $\mathscr{F}_{-g}$. If $\mathbf{c}$ is a boundary in \[\mathscr{F}_{-g}/\mathscr{F}_{-g-1} = \mathscr{F}_{-g}\] then  $\mathbf{c}$ is a boundary in $\CFh(\mathcal{H})$ as well.
\end{proof}

\section{The corollaries}
\label{sec:cor}

We prove Corollaries \ref{cor:prime}, \ref{cor:trefoil}, and \ref{cor:general} below.

\begin{proof}[Proof of Corollary~\ref{cor:prime}]
Suppose for a contradiction that  \[K=K_1\#K_2\] is an $L$-space knot which is  a connected sum of nontrivial knots $K_1$ and $K_2$. We know that $K$ is fibered, which implies that each summand is also fibered by \cite{gabai}. Thus,  \begin{equation}\label{eqn:fiberedki}\HFKh(S^3,K_1,g(K_1))\cong \F\cong \HFKh(S^3,K_1,g(K_2)).\end{equation} The K{\"u}nneth formula for knot Floer homology states that \[\HFKh(S^3,K,i)\cong \bigoplus_{i_1+i_2=i}\HFKh(S^3,K_1,i_1)\otimes\HFKh(S^3,K_2,i_2).\] Since $g(K)=g(K_1)+g(K_2)$, the K{\"u}nneth formula together with \eqref{eqn:fiberedki} implies that \[\HFKh(S^3,K,g(K)-1)\cong \HFKh(S^3,K_1,g(K_1)-1)\oplus \HFKh(S^3,K_1,g(K_2)-1).\] Each of the summands on the right is nontrivial by Theorem~\ref{thm:main}, which implies that \[\rk\HFKh(S^3,K,g(K)-1)\geq 2.\] But this violates the constraint \eqref{eqn:constraint} on the knot Floer homology of $K$, a contradiction.
\end{proof}

\begin{proof}[Proof of Corollary~\ref{cor:trefoil}]
 Suppose  that \[\rk \HFKh(S^3,K)=3.\] Then $\HFKh(S^3,K)$ is supported in Alexander gradings $0$ and $\pm g(K)$ by  symmetry  and  genus detection \eqref{eqn:genus}. Note that $g(K)\geq 1$ since  $K$ is otherwise the unknot and \[\rk\HFKh(S^3,K)=1,\] a contradiction. So we have that \[\HFKh(S^3,K,i)\cong \begin{cases}
\F,&i=g(K),\\
\F,&i=0,\\
\F,&i=-g(K).
\end{cases}\] Fiberedness detection \eqref{eqn:fibered} therefore implies that $K$ is fibered.  Theorem~\ref{thm:main} then forces $g(K)=1$. Thus, $K$ is a genus one fibered knot. It follows that $K$ is either a trefoil or the figure eight, but  the knot Floer homology of the latter has rank 5, so $K$ is a trefoil.
\end{proof}

\begin{proof}[Proof of Corollary~\ref{cor:general}] Suppose that \begin{equation}\label{eqn:rank}\rk \HFKh(Y,K)=3.\end{equation} Since $K$ is nullhomologous, there is a spectral sequence with $E_2$ page $\HFKh(Y,K)$ and abutting  to $\HFh(Y)$. It follows that \[\rk\HFh(Y)=1\textrm{ or }3.\] In particular, this implies  that $Y$ is a rational homology 3--sphere.

Suppose for a contradiction that $g(K)=0$.  Then $Y$ cannot be $ S^3$ since that would violate our rank assumption \eqref{eqn:rank}. The knot $K$ is therefore contained in a 3--ball in $Y$ whose boundary does not bound a ball in $Y\smallsetminus K$, violating the irreducibility of the knot complement.

So  $g(K)\geq 1$. Since $K$ has irreducible complement and $Y$ is a rational homology 3--sphere, analogues of the genus and fiberedness detection \eqref{eqn:genus} and \eqref{eqn:fibered}  hold \cite{yi-thurston,yi-fibered}. In particular, Theorem~\ref{thm:main} forces $K$ to be a genus one fibered knot exactly as in the previous proof. 

Observe that $\rk\HFKh(Y,K) \neq \rk\HFh(Y)$ since otherwise the monodromy of $K$ is trivial and $Y \cong \#^2(S^1\times S^2)$ by Remark~\ref{rmk:idtrivial}. Thus, \[\rk\HFh(Y)=1,\] meaning that $Y$ is an integer homology 3--sphere and an $L$-space. The genus one fibered knots in $L$-spaces were classified in \cite[Theorem 4.1]{Ba2} in terms of their monodromies. From that classification, it is also easy to determine which such $L$-spaces are  integer homology 3--spheres. In particular,  $K$ must a fibered knot with fiber a once-punctured torus $T$ with monodromy $\varphi$ one of the following, where $x$ and $y$ represent positive Dehn twists around dual curves in $T$:
\begin{enumerate}
\item $\varphi = xy^{-1}$
\item $\varphi = (xy)^{-3}x^{-3}y^{-1} = (xy)^{-6}xy$
\item $\varphi = x^{-1}y^{-1}$
\item $\varphi = (xy)^{3}x^{-3}y^{-1} = xy$
\item $\varphi=(xy)^6 x^{-1}y^{-1}.$
\end{enumerate}
In the first case, $K$ is the figure eight, a contradiction since $\HFKh$ of the figure eight has rank $5$. In the next four cases, $K$ is, respectively:
\begin{itemize}
\item the core of $-1$-surgery on the right-handed trefoil,
\item the left-handed trefoil,
\item the core of $+1$-surgery on the left-handed trefoil,
\item the right-handed trefoil.
\end{itemize}
Each of these also has $\HFKh$ of rank $3$, as computed in \cite{Ba2}, for example. This completes the proof of Corollary~\ref{cor:general}.
\end{proof}

\bibliographystyle{myalpha}
\bibliography{References}

\end{document}